\documentclass{amsart}

\usepackage[backref]{hyperref}
\usepackage{amssymb}
\usepackage{fancyhdr}

\makeatletter
\let\@evenhead\relax
\let\@oddhead\relax
\makeatother

\usepackage{color}
\usepackage[T1]{fontenc}
\usepackage[latin1]{inputenc}
\usepackage{url}
\usepackage{graphicx}
\usepackage{amsmath}
\usepackage{amsthm}
\usepackage{ae}

\newtheorem{theo}{Theorem}[section]
\newtheorem{cor}[theo]{Corollary}
\newtheorem{lemm}[theo]{Lemma}

\newtheorem{prop}[theo]{Proposition}

\def\C{\mathbb{C}}
\def\R{\mathbb{R}}

\def\H{\mathbb{H}}

\def\GL{{\sf{GL}}}

\def\SO{{\sf{SO}}}

\def\Sp{{\sf{Sp}}}

\def\sl{{\mathfrak{sl}}}

\def\p{{\mathfrak{p}}}

\def\a{{\mathfrak{a}}}
\def\g{{\mathfrak{g}}}

\def\so{{\mathfrak{so}}}
\def\sp{{\mathfrak{sp}}}

\def\h{{\mathfrak{h}}}
\def\i{{\mathfrak{i}}}

\def\q{{\mathfrak{q}}}
\def\l{{\mathfrak{l}}}

\def\m{{\mathfrak{m}}}
\def\su{{\mathfrak{su}}}

\def\Ein{{\mathsf{Eins}}}
\title[]{Pseudo-Conformal actions of semisimple Lie groups}

\date{\today}
\author{Mehdi Belraouti}
\address{Mehdi Belraouti \newline
Facult\'e de Math\'ematiques,\\
USTHB, BP 32, El-Alia,\\
16111 Bab-Ezzouar, Alger (Algeria)}
\email{mbelraouti@usthb.dz}

\author{Mohamed Deffaf}
\address{Mohamed Deffaf \newline
Facult\'e de Math\'ematiques,\\
USTHB, BP 32, El-Alia,\\
16111 Bab-Ezzouar, Alger (Algeria)}
\email{mdeffaf@usthb.dz}

\author{Abdelghani Zeghib}
\address{Abdelghani Zeghib \newline
UMPA, ENS de Lyon, France }
\email{abdelghani.zeghib@ens-lyon.fr}

\begin{document}
\maketitle

\noindent{\bf Abstract.}
We consider the pseudo-Riemannian Lichnerowicz conjecture in the homogeneous setting. In particular, we show that any compact connected pseudo-Riemannian manifold $M$ on which a semisimple group $G$ acts conformally, essentially and transitively,  is conformally flat.

\tableofcontents

\section{Introduction}
The present article deals with the so-called pseudo-Riemannian Lichnerowicz conjecture, formulated by D'Ambra and Gromov (see \cite{Gromov}) as the pseudo-Riemannian analogue of the original Lichnerowicz conjecture in the Riemannian setting. It asserts that if a compact pseudo-Riemannian conformal manifold is essential, then it is conformally flat. Unlike the Riemannian case, where the conjecture was proved independently by Obata and Ferrand \cite{Obata,Ferrand}, the pseudo-Riemannian analogue was disproved by Frances \cite{Francesun}. In particular, for signature with $p\geq 2$, Frances constructed a two-parameter family of non-conformally flat $(p,q)-$analytic pseudo-Riemannian essential structures on $\mathbb{S}^{1}\times \mathbb{S}^{p+q-1}$. Neverthless, the question still open in the Lorentzian case. In \cite{frances2}, Frances and Melnick proved the Lichnerowicz conjecture for real-analytic three-dimensional Lorentzian manifolds. In the Lorentzian case, D'Ambra and Gromov further conjectured that conformal essential manifolds are, up to finite cover, conformally equivalent to the Einstein space $\operatorname{Ein}^{1,n}$. This conjecture  was also dispoved  by Frances  in \cite{Francestrois}, where he constructed infinitely many distinct essential, conformally flat Lorentzian structures on the product of $\mathbb{S}^{1}$ by the connected sum of $m$ copies of $\mathbb{S}^{1}\times \mathbb{S}^{p+q-1}$. 

The work of Frances shows that the essentiality of the conformal group is not as restrictive as in the riemannian case, and thus it is not reasonable to expect a full classification in the pseudo-Riemannian setting. A more restrictive framework consists in considering manifolds $M$ with the conformal action of an essential  group $G$ that is algebraically interesting. This was adopted in the works of Zimmer, Bader, Zeghib, Frances and Pecastaing (see \cite{Zimmerun}, \cite{Bader}, \cite{Francesquatre}, \cite{Pecastaingun}, \cite{Pecastaingdeux}). A direct consequence of \cite{Zimmerun} is that, when the group $G$ is assumed to be simple, its  real rank is less than $p+1$ (See \cite{Bader}). In \cite{Bader} and \cite{Francesquatre} the maximal rank case is studied. They proved, in particular, that the simple essential group $G$ is locally isomorphic to $\SO(p+1,k)$, for some $p+1\leq k\leq q+1$, and that the manifold is conformally flat. Yet, the maximal real-rank hypothesis restricts considerablly the geometry. In the general case and even with the simplicity hypothesis, a wide variety of examples exists. In particular, all the examples constructed in \cite{Francestrois} admit an essential conformal action of $\operatorname{PSL}(2,\mathbb {R})$. Pecastaing \cite{Pecastaingun} proved that if the group $G$ is locally isomorphic to $\operatorname{PSL}(2,\mathbb {R})$, then the  manifold $M$ is conformally flat. In \cite{Pecastaingdeux}, Pecastaing considered the minimal-rank case and determined the smallest possible value of the index $(p,q)$ of the pseudo riemannian metric. He proved, in particular, that when the index is optimal, the manifold $M$ is conformally flat.

A natural and legitimate question is to consider the pseudo-Riemannian Lichnerowicz conjecture in a homogeneous setting; that is, to study a compact, connected pseudo-Riemannian manifold $M$ on which a Lie group $G$ acts conformally, essentially, and transitively. The present work constitutes a third step in addressing this question \cite{BDRZ1, Deffaf2025}. In the first paper \cite{BDRZ1}, the question was answered positively under the assumption that the non-compact semisimple part is the Mobius group, and a complete classification theorem was established. In the second article \cite{Deffaf2025}, the problem was considered in the complex homogeneous setting; that is, assuming that $M$ is a compact complex manifold endowed with a conformal class of holomorphic Riemannian metrics, and $G$ is a complex semisimple Lie group. Here too, the conjecture was confirmed, and a full classification theorem was provided.

While the settings considered in the first two articles were highly restrictive, they nevertheless constitute crucial steps toward the general homogeneous situation studied in this paper, in which we assume only that the group \(G\) is semisimple. Both can be seen as technical components of the overall proof. Our results make optimal use of the classification of simple non-compact real Lie algebras. It is unclear to us how this result could be obtained without relying on it.

Recall that a conformal pseudo-Riemannian manifold is a differentiable manifold $M$ endowed with a conformal class $[g] = \left\lbrace \exp(f)g \;\middle|\; 
f : M \to \mathbb{R}, \; f \text{ is a } C^\infty \text{ function} \right\rbrace
$ of a pseudo-Riemannian metric $g$ of signature $(p,q)$. Up to a twofold covering, the Einstein universe $\operatorname{Ein}^{p,q}$ is defined as the product $\mathbb{S}^{p}\times\mathbb{S}^{q}$, endowed with the conformal class of $-g_{\mathbb{S}^{p}}\oplus g_{\mathbb{S}^{q}}$, where $g_{\mathbb{S}^{p}}$ and $g_{\mathbb{S}^{q}}$ are the standard Riemannian metrics on the spheres $\mathbb{S}^{p}$ and $ \mathbb{S}^{q}$, respectively. It is a compact conformally flat manifold, that is, locally conformally diffeomorphic to  Minkowski space $\mathbb{R}^{p,q}$. The group $\operatorname{PO}(p+1,q+1)$, which turns out to be the conformal group of $\operatorname{Ein}^{p,q}$,  acts transitively on it. Additionally, by Liouville's theorem, conformal local diffeomorphisms of $\operatorname{Ein}^{p,q}$ are precisely the restrictions of elements of $\operatorname{PO}(p+1,q+1)$. Combined with the fact that Minkowski space can be embedded conformally as a dense open subset of the Einstein universe $\operatorname{Ein}^{p,q}$, this shows that a manifold is conformally flat if and only if it admits a $\left(\operatorname{PO}(p+1,q+1), \operatorname{Ein}^{p,q}\right)-$structure. 

A subgroup $G$ of the conformal group $\operatorname{Conf}(M,g)$ is said to act essentially if there is no metric in the conformal class $[g]$ preserved by $G$. The action of $\operatorname{PO}(p+1,q+1)$ on $\operatorname{Ein}^{p,q}$ provides a fundamental example of an essential action.

In this paper we prove:
\begin{theo}
\label{t}
Let $M$ be a compact, connected pseudo-Riemannian manifold.  
Let $G$ be a semisimple group which acts conformally, essentially, and transitively on $M$.  
Then $M$ is conformally flat.
\end{theo}

Assume that $M$ is a compact, simply connected pseudo-Riemannian manifold endowed with a transitive and essential action of the conformal group $\operatorname{Conf}(M)$. Then, by Montgomery's theorem~\cite[Theorem~A]{montgomery1950simply}, there exists a compact subgroup $K \subset \operatorname{Conf}(M)$ acting transitively on $M$. Since $M$ is simply connected, one may assume that $K$ is contained in the semisimple part of a Levi decomposition of $\operatorname{Conf}(M)$. Thus, $M$ is homogeneous under the action of a semisimple Lie group $G$. Moreover, by~\cite[Proposition~2.3]{BDRZ1}, $G$ acts essentially on $M$. As a consequence, we obtain the following corollary:

\begin{cor}
\label{corcorcorcor}
Let $M$ be a compact, simply connected pseudo-Riemannian manifold. Assume that the conformal group $\operatorname{Conf}(M)$ acts essentially and transitively on $M$. Then $M$ is conformally equivalent to $\operatorname{Ein}^{p,q}$. 
\end{cor}

Finally, let us observe that Theorem~\ref{t} applies to compact semisimple Lie groups endowed with a left-invariant pseudo-Riemannian conformal structure. Such groups are not necessarily  simply connected, but they do have a finite fundamental group, and so their universal cover is compact. Further details may be found in \cite{zeghib2025semiriemannianmetricscompactsimple}.

\subsection{Organization of the  article}
In Section~\ref{sect1}, we introduce some general facts and provide an algebraic formulation of our initial problem in terms of Lie algebra terminology. Section~\ref{secexample} is devoted to the study of an important example of the action of $\Sp(p,q)$ on $\Ein^{4p-1, 4q-1}$. In Section~\ref{sect?}, we introduce the Heredity Principle and use it to obtain information about the root decomposition of our Lie algebra. Finally, in Sections~\ref{sect3} and \ref{sect4}, we prove our main theorem.

\section{Preliminaries}
\label{sect1}
Let $(M,g)$ be a compact, connected pseudo Riemannian manifold of dimension $n$, endowed with a transitive and essential conformal action of a semisimple Lie group $G$.  The Lie algebra of $G$ will be denoted by $\mathfrak{g}$. By \cite[Lemma~2.1]{BDRZ1}, we may assume without loss of generality that $G$ is connected. 

By \cite[Proposition~2.3]{BDRZ1}, the semisimple algebra $\mathfrak{g}$ is non-compact. Let then $\mathfrak{a}$ be a Cartan subalgebra of $\mathfrak{g}$ with respect to some Cartan involution $\Theta$. Consider $\mathfrak{g}=\mathfrak{g}_{0}  \bigoplus_{\alpha \in \Delta} \mathfrak{g}_{\alpha}=   \mathfrak{a}\oplus \mathfrak{m} \bigoplus_{\alpha \in \Delta} \mathfrak{g}_{\alpha}$ the associated root space decomposition of $\mathfrak{g}$, where $\Delta$ is the set of roots of $(\mathfrak{g},\mathfrak{a})$. Let $\Delta^{+}$, $\Delta^{-}$ be respectively the set of  positive and negative roots of $\mathfrak{g}$ for some choosen notion of positivity on $\mathfrak{a}^{*}$, and let $\Pi\subset \Delta^{+}$, be the basis of the root system $\Delta$. Then $\mathfrak{g}=  \mathfrak{g}_{-}\oplus\mathfrak{a}\oplus \mathfrak{m} \oplus \mathfrak{g}_{+}$, where $\mathfrak{g}_{+}=\bigoplus_{\alpha\in \Delta^{+}} \mathfrak{g}_{\alpha}$ and $\mathfrak{g}_{-}=\bigoplus_{\alpha\in \Delta^{-}} \mathfrak{g}_{\alpha}$. 
\subsection{Parabolic subalgebras}
\label{propppppppppppppppdefffff}
A subalgebra $\q$ of $\g$ containing  the minimal parabolic subalgebra $\mathfrak{a}\oplus \mathfrak{m} \oplus \mathfrak{g}_{+}$ is said to be parabolic. By \cite[Proposition~7.76]{K}), any such subalgebra is of the form  $\bigoplus_{\alpha\in \operatorname{span}(-\Pi')}\mathfrak{g}_{\alpha}\oplus\mathfrak{a}\oplus \mathfrak{m} \oplus \mathfrak{g}_{+}$, for some subset $\Pi'$ of $\Pi$. Let $\i$ be an ideal of $\q$. 
\begin{prop}
\label{proppppppppppppppp}
If $\i$ contains $\mathfrak{a}\oplus \mathfrak{m}$, then $\i=\q$.
\end{prop}
\begin{proof}
We have that $\bigoplus_{\alpha\in \operatorname{span}(-\Pi')}\mathfrak{g}_{\alpha} \oplus \mathfrak{g}_{+}=[\a,\bigoplus_{\alpha\in \operatorname{span}(-\Pi')}\mathfrak{g}_{\alpha} \oplus \mathfrak{g}_{+}]\subseteq [\i,\q]\subseteq \i $, where the last inclusion holds because $\i$ is an ideal of $\q$. Therefore, $\q=\mathfrak{a}\oplus \mathfrak{m}\oplus \bigoplus_{\alpha\in \operatorname{span}(-\Pi')}\mathfrak{g}_{\alpha} \oplus \mathfrak{g}_{+}\subseteq \i+\i=\i$, and hence $\q=\i$.
\end{proof}
\subsection{Gauss map}
Let $\operatorname{Sym}(\mathfrak{g})$ denote the space of symmetric bilinear form on $\mathfrak{g}$, endowed with the natural action of $G$ given by $g.B(X,Y)=B(Ad_{g^{-1}}X,Ad_{g^{-1}}Y)$. We define a map $\Phi$, called the Gauss map, from $M$ to $\operatorname{Sym}(\mathfrak{g})$ by
 $\Phi(x)(X,Y)=g_x( \overline{X}_x,\overline{Y}_x)$, where $\overline{X}, \overline{Y}$ are the fundamental vector fields associated to $X$ and $Y$. This map $\Phi$ is equivariant if and only if $G$ acts isometrically on $M$. When the action is conformal, the progective map $\mathbb{P}(\Phi): M\to \mathbb{P}(\operatorname{Sym}(\mathfrak{g}))$ is equivariante. Denote by $\langle,\rangle$ the image of $x$ under $\Phi$; it is a degenerate bilinear form on $\mathfrak{g}$ whose kernel is $\mathfrak{h}$.
\subsection{Distortion map}
Denote by $P$ the subgroup of $G$ preserving the conformal class of $\langle,\rangle$. On the one hand, it is an Ad-algebraic group  normalizing $H$; in particular, it contains the smallest Ad-algebraic subgroup $H^*=Ad^{-1}(\overline{Ad(H)})$ that contains $H$. By \cite[Fact~3.1]{BDRZ1}, up to a conjugacy, the Lie algebra $\mathfrak{p}$ of $P$ contains the maximal triangular subalgebra $\mathfrak{a}\oplus \mathfrak{g}_+$.

On the other hand, there is a homomorphism $\delta:P\longrightarrow \mathbb{R}$ satisfying: 
\begin{equation}
\label{Euqa1}
\langle Ad_{p}(u),Ad_{p}(v)\rangle=\exp(\delta(p))\langle u,v\rangle \quad \text{for every } p \in P \text{ and } u, v \in \mathfrak{g}.
\end{equation}
Differentiating  Equation \ref{Euqa1} gives a homomorphism, which we will continue to denote $\delta$, from $\mathfrak{p}$ to $\mathbb{R}$ such that:
\begin{equation}
\label{Equa2}
\langle ad_p u,v\rangle +\langle u,ad_p v\rangle=\delta(p)\langle u,v\rangle \quad \text{for every } p \in \mathfrak{p} \text{ and } u, v \in \mathfrak{g}.
\end{equation}
The restriction of $\delta$ to $\mathfrak{a}$, which is non-trivial by \cite[Proposition~3.2]{BDRZ1}, is the distortion map. 
\subsection{Pairing condition}
Two elements $\alpha$ and $\beta$ of $\Delta \cup {0}$ are said to be paired if $\mathfrak{g}_{\alpha}$ and $\mathfrak{g}_{\beta}$ are not $\langle .,.\rangle$-orthogonal. Since $\langle .,.\rangle$ is non-trivial, there are at least two paired elements $\alpha$ and $\beta$ in $\Delta \cup {0}$. In this case, $\alpha + \beta = \delta$. This shows, in particular, that if $\alpha$ is a root such that $\mathfrak{g}_{\alpha} \not\subseteq \mathfrak{h}$, then $\delta - \alpha$ is also a root such that $\mathfrak{g}_{\delta-\alpha} \not\subseteq \mathfrak{h}$. We refer to this property as the pairing condition.
 
 \subsection{Equivalent problem in the Lie algebra level} 
 \label{subsect14}Now we will forget our original problem and argue at a pure Lie algebra level. We say that the triple $(\mathfrak{g},\mathfrak{h},\langle,\rangle)$ verify the algebraic formulation if the following properties hold:
 \begin{itemize}
 	\item The non-compact semisimple part of $\mathfrak{g}$ is a non-trivial, $\mathfrak{h}$ a subalgebra of $\mathfrak{g}$, and $\langle,\rangle$ is a degenerate scalar product on $\mathfrak{g}$ whose kernel is $\mathfrak{h}$;
 	\item There is a root space decomposition as described above;
 	\item There is a non-trivial distortion $\delta:\mathfrak{a}\longrightarrow \mathbb{R}$;
 	\item The pairing condition is satisfied;
 	\item The isotropy subalgebra normalizes $\mathfrak{a}\oplus \mathfrak{g}_+$.
 \end{itemize}
\subsection{Modification}
We say that a subalgebra $\mathfrak{g}'$ is a modification of $\mathfrak{g}$, if $\mathfrak{g}'$ projects surjectively on $\mathfrak{g}/\mathfrak{h}$. Equivalently, $G'/(G'\cap H)$ is an open orbit of $M$, where  $G'$ is the connected subgroup of $G$ associated to $\mathfrak{g}'$.
\subsection{Hypothesis}
\label{Hyp}
Note that if $\g_{0}\subset\h$, then by \cite[Proposition~3.3]{BDRZ1}, $M=G/H$ is conformally flat. From now on we therefore assume that $\g_{0}\nsubseteq\h$. In this case, $\delta$ is a root and $\g_{\delta}$ is paired with $\g_{0}$. We have the following important Lemma:
\begin{lemm}
\label{Lemma1}
Up to modification, the non-compact semisimple part of $\g$ is simple.
\end{lemm}
\begin{proof}
Assume that the non-compact semisimple part of $\mathfrak{g}$ decomposes as a direct sum
$
\mathfrak{g}_{1} \oplus \mathfrak{g}_{2},
$
where $\mathfrak{g}_{1}$ is simple with $\mathfrak{g}_{1}\nsubseteq \mathfrak{h}$, and $\mathfrak{g}_{2}$ is semisimple with $\mathfrak{g}_{2}\nsubseteq \mathfrak{h}$. Since the root system of $\mathfrak{g}$ is the disjoint union of the root systems of $\mathfrak{g}_{1}$ and $\mathfrak{g}_{2}$, we may assume, without loss of generality, that $\delta$ is a root of $\mathfrak{g}_{1}$.

Suppose now that there exists a root $\alpha$ of $\mathfrak{g}_{2}$ such that $\mathfrak{g}_{\alpha}\nsubseteq \mathfrak{h}$. Then $\delta - \alpha$ would be a root of $\mathfrak{g}$. However, since $\delta$ belongs to the root system of $\mathfrak{g}_{1}$ and $\alpha$ belongs to that of $\mathfrak{g}_{2}$, and these two systems are disjoint, this would force either $\delta$ or $\alpha$ to lie in both root systems, a contradiction.

Thus, for every root $\alpha$ of $\mathfrak{g}_{2}$, we must have $\mathfrak{g}_{\alpha} \subseteq \mathfrak{h}$, and therefore $\mathfrak{g}_{2}\subseteq \mathfrak{h}$. We conclude that, after  modification, the non-compact semisimple part of $\mathfrak{g}$ must in fact be simple.

\end{proof}

\section{Action of $\Sp(p,q)$ on $\Ein^{4p-1, 4q-1}$}
\label{secexample}
Let \(\mathbb{H}^{p,q}\) denote the \(p+q\)-dimensional quaternionic space \(\mathbb{H}^{p+q}\) with \(p \leq q\), endowed with the standard Hermitian quadratic form $$
 \q(u_1, \ldots, u_p, v_1, \ldots, v_q) = \sum_{i=1}^p |u_i|^2 - \sum_{j=1}^q |v_j|^2.
$$

Consider \(\Sp(p,q)\), the subgroup of \(\GL(p+q,\mathbb{H})\) preserving the quaternionic Hermitian form on \(\mathbb{H}^{p,q}\). Via the natural identification \(\mathbb{H}^{p,q} \simeq \mathbb{R}^{4p, 4q}\), this group can be viewed as a subgroup of \(\SO(4p, 4q)\) whose elements commute with the right action of \(\mathbb{H}\) on \(\mathbb{H}^{p,q}\), defined by
\[
h \cdot (u_1, \ldots, u_p, v_1, \ldots, v_q) = (u_1 h, \ldots, u_p h, v_1 h, \ldots, v_q h).
\]

Denote by \(\mathcal{C}_0 = \{ z \in \mathbb{H}^{p,q} \mid q(z) = 0 \}\) the light cone associated to \(\q\) and let \(\mathcal{Q}_{p+q}(\mathbb{H})\) be its quaternionic projectivization \(\mathbb{H}^* \diagdown \mathcal{C}_0 \subset \mathbb{H}\mathbb{P}^{p+q-1}\), where \(\mathbb{H}\mathbb{P}^{p+q-1}\) is the right projective quaternionic space.  
Similarly, we consider \(\mathcal{Q}_{p+q}(\mathbb{R})\) to be the real projectivization \(\mathcal{C}_0 / \mathbb{R}^* \subset \mathbb{R}\mathbb{P}^{4(p+q)-2}\) of \(\mathcal{C}_0\). Under the identification \(\mathbb{H}^{p,q} \simeq \mathbb{R}^{4p, 4q}\), the standard Hermitian quadratic form \(\q\) defines a conformal structure  on \(\mathcal{Q}_{p+q}(\mathbb{R})\), which is then naturally identified with the Einstein universe \(\Ein^{4p-1, 4q-1}\). In this case, the natural projection $\pi:\Ein^{4p-1, 4q-1}\simeq \mathcal{Q}_{p+q}(\mathbb{R})\longrightarrow \mathcal{Q}_{p+q}(\mathbb{H})$ is a fibration  whose fibers are topologically identified with $\mathbb{R}P^{3}$. It is equivariant under the action of \(\Sp(p,q)\).

The group \(\Sp(p,q)\) acts transitively on both \(\mathcal{Q}_{p+q}(\mathbb{H})\) and \(\Ein^{4p-1, 4q-1}\). Let \(z \in \mathcal{C}_0\), and denote by \(\operatorname{P}^{\mathbb{H}}\) and \(\operatorname{P}^{\mathbb{R}}\) the stabilizers of \(z\mathbb{H} \in \mathcal{Q}_{p+q}(\mathbb{H})\) and \(\mathbb{R}z \in \Ein^{4p-1, 4q-1}\), respectively.  Let us also denote by $\operatorname{P}$ the stabilizer of \(\mathbb{R}z \in \Ein^{4p-1, 4q-1}\) in \(\SO(4p, 4q)\). Finally, let $\mathfrak{p}^{\mathbb{R}}$, $\mathfrak{p}^{\mathbb{H}}$ and $\mathfrak{p}$ denote the Lie subalgebras associated with \(\operatorname{P}^{\mathbb{R}}\), \(\operatorname{P}^{\mathbb{H}}\), and $\operatorname{P}$, respectively.

On the one hand, \(\operatorname{P}^{\mathbb{H}}\) acts tansitively on the fiber $\pi^{-1}(z\mathbb{H})$. On the other hand, since the elements of \(\operatorname{P}^{\mathbb{H}}\) commute with the right action of  \(\mathbb{H}\) on \(\mathbb{H}^{p,q}\), we obtain that \(\operatorname{P}^{\mathbb{R}}\) acts trivially on the fiber $\pi^{-1}(z\mathbb{H})$, and hence \(\operatorname{P}^{\mathbb{R}}\) is a normal subgroup of \(\operatorname{P}^{\mathbb{H}}\). Consequently, $\mathfrak{p}^{\mathbb{R}}$ is an ideal of $\mathfrak{p}^{\mathbb{H}}$, and $\mathfrak{p}^{\mathbb{H}}$ decomposes as  a direct sum of vector spaces $\mathfrak{p}^{\mathbb{H}}=\mathfrak{p}^{\mathbb{R}}\oplus \so(3)$.

Let us finally note that \(\mathfrak{p}\) is a parabolic subalgebra (in the sense of subsection~\ref{propppppppppppppppdefffff}). 
Now consider a root space decomposition
$
\mathfrak{g}_{-} \oplus \mathfrak{a} \oplus \mathfrak{m} \oplus \mathfrak{g}_{+}
$
of \(\mathfrak{sp}(p,q)\). 
The nilpotent subalgebra \(\mathfrak{a}\oplus\mathfrak{g}_{+}\) lies, inside \(\mathfrak{so}(4p,4q)\), in a Borel subalgebra and hence in a minimal parabolic one. 
Since the latter is unique up to conjugacy, we may assume that 
$
\mathfrak{a}\oplus\mathfrak{g}_{+}\subset \mathfrak{p}.
$
But \(\mathfrak{p}^{\mathbb{R}} \subset \mathfrak{p}\), and therefore 
$
\mathfrak{a}\oplus\mathfrak{g}_{+}\subset \mathfrak{p}^{\mathbb{R}}.
$

\section{Heredity principale} 
\label{sect?}
Let $\Gamma$ be a closed symmetric subset of $\Delta$, , that is, 
if $\alpha, \beta \in \Gamma$ and $\alpha + \beta \in \Delta$, then $\alpha + \beta \in \Gamma$, 
and if $\alpha \in \Gamma$, then $-\alpha \in \Gamma$. Consider the $\Theta$ invariant Lie subalgebra of $\mathfrak{g}$ defined by $\g_{\Gamma}=\g_{0} \bigoplus_{\alpha \in \Gamma}\g_{\alpha}$. By \cite[Corollary~6.29]{K}, this is a reductive Lie subalgebra of $\mathfrak{g}$. Moreover, both the semisimple part $[\g_{\Gamma}, \g_{\Gamma}]$ and the center $Z(\g_{\Gamma})$ of $\g_{\Gamma}$ are preserved by $\Theta$. 
 
Let $\mathfrak{g} = \mathfrak{p}_0 \oplus \mathfrak{l}_0$ be the Cartan decomposition of $\mathfrak{g}$ associated with $\Theta$. The subspaces $\mathfrak{p}_0$ and $\mathfrak{l}_0$ are orthogonal with respect to the Killing form $B$ of $\mathfrak{g}$, which is positive definite on $\mathfrak{p}_0$ and negative definite on $\mathfrak{l}_0$. Since $\mathfrak{g}_0 = \mathfrak{a} \oplus \mathfrak{m} \neq 0$, the Killing form restricts to a nondegenerate scalar product on $\mathfrak{g}_\Gamma$, and it follows that the involution $\Theta$ induces a Cartan involution on the subalgebra $[\mathfrak{g}_\Gamma, \mathfrak{g}_\Gamma]$. The corresponding Cartan decomposition is then given by
$
[\mathfrak{g}_{\Gamma}, \mathfrak{g}_{\Gamma}] = (\mathfrak{p}_0 \cap [\mathfrak{g}_{\Gamma}, \mathfrak{g}_{\Gamma}]) \oplus (\mathfrak{l}_0 \cap [\mathfrak{g}_{\Gamma}, \mathfrak{g}_{\Gamma}]).
$ In fact we have the following proposition: 
\begin{prop}
\label{pro0001}
$\a\cap [\g_{\Gamma}, \g_{\Gamma}]$ is a Cartan subalgebra of $ [\g_{\Gamma}, \g_{\Gamma}]$ associated with $\Theta$.
\end{prop}
\begin{proof}
All we need to prove is that $\mathfrak{a}\cap[\mathfrak{g}_{\Gamma},\mathfrak{g}_{\Gamma}]$ is a maximal abelian subspace of $\mathfrak{p}_0\cap[\mathfrak{g}_{\Gamma},\mathfrak{g}_{\Gamma}]$. Assume this is not the case; then there exists an element $a'\in\mathfrak{p}_0\cap[\mathfrak{g}_{\Gamma},\mathfrak{g}_{\Gamma}]\setminus\mathfrak{a}$ that commutes with $\mathfrak{a}\cap[\mathfrak{g}_{\Gamma},\mathfrak{g}_{\Gamma}]$. Note that $a'$ cannot lie in $\mathfrak{g}_0$, since otherwise it would commute with all of $\mathfrak{a}$, contradicting the maximality of $\mathfrak{a}$ in $\mathfrak{p}_0$. Therefore, there is $\alpha \in \Gamma$ such that the projection of $a'$ on $\g_{\alpha}$ is non zero.

Now, let $H_{\alpha}$ denote the unique element of $\mathfrak{a}$ such that
$
B_{\Theta}(H_{\alpha},\cdot)=\alpha,
$
where $B_{\Theta}(X,Y)=-B(X,\Theta(Y))$. By \cite[Proposition~6.52]{K}, we have $H_{\alpha}\in[\mathfrak{g}_{\alpha},\mathfrak{g}_{-\alpha}]$ and hence $H_{\alpha}\in\mathfrak{a}\cap[\mathfrak{g}_{\Gamma},\mathfrak{g}_{\Gamma}]$. But $H_{\alpha}$ does not commute with $a'$, since $[H_{\alpha},a']$ has a nontrivial projection onto $\mathfrak{g}_{\alpha}$. This is a contradiction.
\end{proof}

Note that, for every $\alpha \in \Gamma$, we have $\mathfrak{g}_{\alpha}=[\mathfrak{a},\mathfrak{g}_{\alpha}]\subseteq [\mathfrak{g}_{\Gamma},\mathfrak{g}_{\Gamma}]$. Moreover,  $\mathfrak{g}_{0}$ decomposes as the direct sum
\begin{equation}
\label{equaajout}
\mathfrak{g}_0 = (\mathfrak{g}_0 \cap [\mathfrak{g}_\Gamma, \mathfrak{g}_\Gamma]) \;\oplus\; (\mathfrak{g}_0 \cap Z(\mathfrak{g}_\Gamma))
\end{equation}

Together with Proposition~\ref{pro0001}, this shows that the root space decomposition of $[\mathfrak{g}_{\Gamma},\mathfrak{g}_{\Gamma}]$ is
\begin{equation}
\label{equaajoutt}
[\mathfrak{g}_{\Gamma},\mathfrak{g}_{\Gamma}]
= (\mathfrak{g}_{0}\cap[\mathfrak{g}_{\Gamma},\mathfrak{g}_{\Gamma}]) \;\bigoplus_{\alpha \in \Gamma} \mathfrak{g}_{\alpha}
= (\mathfrak{a}\cap[\mathfrak{g}_{\Gamma},\mathfrak{g}_{\Gamma}]) \oplus (\mathfrak{m}\cap[\mathfrak{g}_{\Gamma},\mathfrak{g}_{\Gamma}]) \;\bigoplus_{\alpha \in \Gamma} \mathfrak{g}_{\alpha}.
\end{equation}

If $\delta \in \Gamma$, the pairing condition implies that $\g_{\Gamma}$ is orthogonal to $\g_{\beta}$ for every $\beta \in \Delta \setminus \Gamma$. This shows that $\h \cap \g_{\Gamma}$ is the kernel of $\langle ., . \rangle_{\mathfrak{g}_{\Gamma}}$, the restriction of $\langle ., . \rangle$ to $\g_{\Gamma}$. In fact, the triple $(\g_{\Gamma},\h \cap \g_{\Gamma}, \langle,\rangle _{\mathfrak{g}_{\Gamma}})$ satisfies the algebraic formulation given subsection~\ref{subsect14}; in this case we say that $\g_{\Gamma}$ satisfies the heridity principale.
\subsection{The M\"obius subalgebra $\hat{\g}(\delta)$} \label{Mobius subalgebra}
Let $\Gamma_{\delta} =\{\, n\delta \in \Delta \mid n \in \mathbb{Z} \,\}$. The smallest subalgebra satisfying the heredity principle described above is $\g_{\Gamma_{\delta}}$. Denote by $\hat{\g}(\delta)$ its non compact semisimple part. It is a rank-one simple Lie algebra  with the root decomposition $\hat{\g}(\delta) = \hat{\a} \oplus \hat{m} \oplus \g_{2\delta} \oplus \g_\delta \oplus \g_{-\delta} \oplus \g_{-2\delta}$, where $\hat{\a} = \mathfrak{a} \cap \hat{\g}(\delta)=\mathfrak{a}\cap[\g_{\Gamma_{\delta}}, \g_{\Gamma_{\delta}}]$ and $\hat{m} = \mathfrak{m} \cap\hat{\g}(\delta)$. 
\begin{prop}
\label{prop5}
The Lie subalgebra $\hat{\g}(\delta)$ is isomorphic to the M\"obius Lie algebra $\so(1,n+1)$.
\end{prop}

\begin{proof}
All we need to prove is that $2\delta$ is not a root. For that, we proceed as follows:

First, we have $\hat{\a} \subset \h$. Indeed, on the one hand, $\hat{\a} \subset \a \subset \g_{0}$ is orthogonal to every $\g_{\alpha}$ with $\alpha \neq \delta$. On the other hand, let $0 \neq x \in \g_{\delta}$. Since $\hat{\a}$ is one-dimensional, the element $[x, \Theta(x)] \neq 0$ generates $\hat{\a}$. Now, either $x$ or $\Theta(x)$ lies in $\p$, and since both are nilpotent elements, one of them preserves $\langle \cdot, \cdot \rangle$. Applying Equation \ref{Equa2} to $x$ and $\Theta(x)$ yields $\langle [x, \Theta(x)], x \rangle = 0$. Therefore, $\hat{\a}$ is also orthogonal to $\g_{\delta}$.

Second, we have  $\g_{\delta}\nsubseteq \p$ and hence $\delta$ is a negative root. Indeed, otherwise we would have $\g_{\delta}=[\hat{a},\g_{\delta}]\subset [\h,\p]\subset \h$, which is clearly not the case. 

Third, $\g_{-\delta}\subset \h$. Indeed, since $\delta$ is a negative root, it follows that $\g_{-\delta}\subset \p$. Consequently, $\g_{-\delta}=[\hat{a},\g_{-\delta}]\subset [\h,\p]\subset \h$.

Fourth, assume that $2\delta$ is a root. Then $\g_{2\delta}$ is paired with $\g_{-\delta}\subset \h$ and  is therefore contained in $\h$. Now, by \cite{Kraljevic}, we have that $\g_{\delta}=[\g_{2\delta},\g_{-\delta}]\subset [\h,\h]\subset \h$, which contradicts the second point. This shows that $2\delta$ is not a root. 
\end{proof}

Consequently, we obtain the following algebras direct sum decompositions:
\begin{prop}
\label{pripou}
We have: \[
\mathfrak{a} = \hat{\mathfrak{a}} \oplus \left( \mathfrak{a} \cap Z(\mathfrak{g}_{\Gamma_{\delta}}) \right),
\]
\[
\mathfrak{m} = \hat{\mathfrak{m}} \oplus \left( \mathfrak{g}_{c} \oplus \mathfrak{m} \cap Z(\mathfrak{g}_{\Gamma_{\delta}}) \right),
\]
where \(\mathfrak{g}_{c}\) is the compact semisimple factor of $\mathfrak{g}$.
\end{prop}
\begin{proof}
Since $\delta$ is non zero on $\hat{a}$, we have that $\a=\hat{a}\oplus \operatorname{Ker}(\delta).$ But $\operatorname{Ker}(\delta)$ is contained in $\mathfrak{a} \cap Z(\mathfrak{g}_{\Gamma_{\delta}})$. Therefore $
\mathfrak{a} = \hat{\mathfrak{a}} \oplus \left( \mathfrak{a} \cap Z(\mathfrak{g}_{\Gamma_{\delta}}) \right),
$

For the second equation, let $m \in \mathfrak{m}$.  
By Equation~\ref{equaajout}, we may decompose $m$ as
$
m = x_{s} + x_{z},
$
where $x_{s} \in \mathfrak{g}_0 \cap [\mathfrak{g}_\Gamma, \mathfrak{g}_\Gamma]$ and $x_{z} \in \mathfrak{g}_0 \cap Z(\mathfrak{g}_\Gamma)$.  
According to Equation~\ref{equaajoutt}, $x_{s}$ further splits as
$
x_{s} = a_{s} + m_{s},
$
with $a_{s} \in \hat{a}$ and $m_{s} \in \hat{m} \oplus \mathfrak{g}_{c}$.  
Similarly, write $x_{z} = a_{z} + m_{z}$, where $a_{z} \in \mathfrak{a}$ and $m_{z} \in \mathfrak{m}$.  
Combining these relations yields $a_{z} = -a_{s} \in \hat{a}$. If $a_{z}\neq 0$, then 
$
\operatorname{ad}_{a_{z}}(\mathfrak{g}_{\delta}) = \mathfrak{g}_{\delta}.
$
Since $\operatorname{ad}_{m_{z}}$ is antisymmetric, it follows that
$
\operatorname{ad}_{x_{z}}(\mathfrak{g}_{\delta}) = \mathfrak{g}_{\delta}.
$
Therefore, $x_{z}$ cannot belong to $Z(\mathfrak{g}_\Gamma)$, a clear contradiction. Hence, necessarily $a_{z} = a_{s} = 0$, and thus $m$ decomposes as  
$
m = m_{s} + m_{z},
$
where $m_{s} \in \hat{m} \oplus \mathfrak{g}_{c}$ and $m_{z} \in \mathfrak{m} \cap Z(\mathfrak{g}_\Gamma).$

\end{proof}

As a consequence, we obtain:  
\[
\mathfrak{g}_{\Gamma_{\delta}} \simeq \mathfrak{so}(1, n+1) \oplus \mathfrak{l},
\]
where  
\[
\mathfrak{so}(1, n+1) \simeq \hat{\mathfrak{g}}(\delta) = \hat{\mathfrak{a}} \oplus \hat{\mathfrak{m}} \oplus \mathfrak{g}_{\delta} \oplus \mathfrak{g}_{-\delta},\] and
\[\mathfrak{l} = \mathfrak{a} \cap Z(\mathfrak{g}_{\Gamma_{\delta}}) \oplus \left( \mathfrak{g}_{c} \oplus \mathfrak{m} \cap Z(\mathfrak{g}_{\Gamma_{\delta}}) \right).
\]

\subsection{Rank two subalgebras} Let $\alpha$ be a root in $\Delta \setminus \mathbb{R}\delta$, and consider $\Gamma$ to be the smallest closed symmetric subsystem of $\Delta$ containing both $\alpha$ and $\delta$. Denote by $\hat{\g}(\alpha, \delta)$ the non-compact semisimple part of the reductive subalgebra $\g_{\Gamma}$. It is a rank-two semisimple subalgebra that satisfies the heredity principle; it contains the  M\"obius Lie algebra $\hat{\g}_{nc}$ and admits \begin{equation}
\label{eq:commutantGamma1} 
\hat{\g}(\alpha, \delta)
= \left( \mathfrak{g}_{0}\cap\hat{\g}(\alpha, \delta) \right) 
  \bigoplus_{\alpha \in \Gamma} \mathfrak{g}_{\alpha}
= \left( \mathfrak{a}\cap \hat{\g}(\alpha, \delta) \right) 
  \oplus \left( \mathfrak{m}\cap \hat{\g}(\alpha, \delta) \right) 
  \bigoplus_{\alpha \in \Gamma} \mathfrak{g}_{\alpha}
  \end{equation}
as root space decomposition. 

\section{Identification of the algebra $\g$}
\label{sect3}
The idea behind the proof of our main theorem involves a meticulous examination of the subalgebra $\m$, combined with an application of the results from \cite{BDRZ1} and \cite{Deffaf2025}. Specifically, we will consider four different cases:

\begin{enumerate}
\item $\m \subset \h$ and $\l \subset \h$,
\item $\m \subset \h$ and $\l \nsubseteq \h$,
\item $\hat{\m}\nsubseteq \h$
\item $\hat{\m}\subset\h$ and $\m \nsubseteq \h$.
\end{enumerate}
In this section, we first establish our main theorem in the first three cases. In a second step, we identify the Lie algebra arising in the fourth case, postponing the study of conformal flatness in this specific situation to the next section. Before proceeding further, let us highlight the following observation:

\begin{prop}
\label{pppprrrr}
If \(\g_{+} \subset \h\), then \(\delta\) is the minimal root.
\end{prop}

\begin{proof}
From the proof of Proposition \ref{prop5}, we know that \(\g_{\delta} \nsubseteq \p\).

Now, let \(\alpha\) be a positive root such that \(\delta - \alpha\) is also a root. Then, \(\g_{\delta-\alpha} \subset \h\). However, by \cite{Kraljevic}, we have \(\g_{\delta} = [\g_{\delta-\alpha}, \g_{\alpha}]\). This implies that \(\g_{\delta} \subset \h \subset \p\), which leads to a clear contradiction.
\end{proof}

\subsection{First case: $\m \subset \h$ and $\l \subset \h$}
Note that, in this case, $\mathfrak{g}_{0} \subset \mathfrak{h}$, which is already excluded by our Hypothesis~\ref{Hyp}. Indeed, on the one hand, by the proof of Proposition~\ref{prop5}, we know that $\hat{\mathfrak{a}}$ is contained in $\mathfrak{h}$. On the other hand, by Proposition~\ref{pripou},
$
\mathfrak{a}
   = \hat{\mathfrak{a}} \oplus \bigl(\mathfrak{a} \cap Z(\mathfrak{g}_{\Gamma_{\delta}})\bigr).
$
Since $\mathfrak{a} \cap Z(\mathfrak{g}_{\Gamma_{\delta}})\subseteq \mathfrak{l}\subseteq \mathfrak{h}$, we deduce that $\mathfrak{a}\subseteq \mathfrak{h}$.

\subsection{Second case (Complex case): $\m \subset \h$ and $\l \nsubseteq \h$}
Let $\g^{\mathbb{C}}$ denote the complexification of $\g$, a complex semisimple Lie algebra. Let $\h^{\mathbb{C}}$ and $\p^{\mathbb{C}}$ denote the complexifications of the subalgebras $\h$ and $\p$, respectively. Since $\m$ is contained in $ \h$, it follows that $\p^{\mathbb{C}}$ is a parabolic subalgebra of $\g^{\mathbb{C}}$ that normalizes $\h^{\mathbb{C}}$. 

Consider $\langle .,.\rangle^{\mathbb{C}}$, the complexification of $\langle .,.\rangle$, defined as its bilinear extension to $\g^{\mathbb{C}}$. It is a degenerate complex bilinear symmetric form with kernel $\h^{\mathbb{C}}$.

Let  $\delta^{\mathbb{C}}:\p^{\mathbb{C}}\longrightarrow \mathbb{C}$ denote the linear complexification of $\delta:\p\longrightarrow \mathbb{R}$, satisfying  $\delta^{\mathbb{C}}(x+iy)=\delta(x)+i\delta(y)$ for every $x,y\in \p$. Thus, 
\begin{equation}
\label{Equa3}
\langle ad_p u,v\rangle^{\mathbb{C}} +\langle u,ad_p v\rangle^{\mathbb{C}}=\delta^{\mathbb{C}}(p)\langle u,v\rangle^{\mathbb{C}} \quad \text{for every } p \in \mathfrak{p}^{\mathbb{C}} \text{ and } u, v \in \mathfrak{g}^{\mathbb{C}}.
\end{equation}

Since the complexification $\a^{\mathbb{C}}$ of $\a$, is contained in the Cartan subalgebra $\a'$ of $\g^{\mathbb{C}}$, and $\delta$ is non trivial on $\a$, it follows  that $\delta^{\mathbb{C}}$ is also non trivial on $\a'$. In summary, the triple $\left(\mathfrak{g}^{\mathbb{C}},\mathfrak{h}^{\mathbb{C}}, \langle .,.\rangle^{\mathbb{C}}\right) $ satisfies the algebraic formulation as described in \cite{Deffaf2025}.

In \cite{Deffaf2025} we classify the triples $\left(\mathfrak{g}^{\mathbb{C}},\mathfrak{h}^{\mathbb{C}}, \langle .,.\rangle^{\mathbb{C}}\right) $. Since $\mathfrak{a}' \not\subseteq \mathfrak{h}^{\mathbb{C}}$, it follows from \cite[Proposition~3.2, Corollary~4.5 and Corollary~5.3]{Deffaf2025}  that, up to modification, $\mathfrak{g}^{\mathbb{C}}$ is simple and isomorphic either to $\mathfrak{sp}(n, \mathbb{C})$ or to $\mathfrak{sl}(n, \mathbb{C})$. In both cases, the triple $\left(\mathfrak{g}^{\mathbb{C}},\mathfrak{h}^{\mathbb{C}}, \langle .,.\rangle^{\mathbb{C}}\right) $ is uniquely determined. Moreover, according to \cite[Theorem~1.2]{Deffaf2025}, there exists a compact connected homogeneous complex manifold $G^{\mathbb{C}} / H^{\mathbb{C}}$ equipped with a conformally flat  holomorphic Riemannian structure such that the Lie algebras of $G^{\mathbb{C}}$ and $H^{\mathbb{C}}$ are $\mathfrak{g}^{\mathbb{C}}$ and $\mathfrak{h}^{\mathbb{C}}$, respectively. Furthermore, the degenerate complex bilinear symmetric form induced by this structure on $\mathfrak{g}^{\mathbb{C}}$ coincides with $\langle \cdot, \cdot \rangle^{\mathbb{C}}$.

Let $G'$ and $H'$ be the connected real Lie subgroups of $G^{\mathbb{C}}$ and $H^{\mathbb{C}}$ associated with the real Lie subalgebras $\mathfrak{g}$ and $\mathfrak{h}$, respectively. In this context, the homogeneous space $G'/H'$ is a real analytic submanifold of $G^{\mathbb{C}} / H^{\mathbb{C}}$, viewed as a real analytic space. It is endowed with a conformal pseudo-Riemannian structure whose complexification corresponds to the conformal holomorphic Riemannian structure of $G^{\mathbb{C}} / H^{\mathbb{C}}$. In particular, $G'/H'$ is also conformally flat. Moreover, one can check that $G'/H'$ is a covering space of $G/H$, it follows that $G/H$ is conformally flat as well.

\subsection{Third case: $\hat{\m} \subseteq \h$ and  $\m\nsubseteq \h $}
\label{subsection5.2}
Since $\hat{\m} \subseteq \h$,  by \cite[Proposition~4.1]{BDRZ1} (more precisely, its proof), we have $\h \cap \g_0=\left( \h \cap \mathfrak{g}_{\Gamma_{\delta}}\right)\cap \g_0$ has codimension one in $\g_0$. Since $\m \nsubseteq \h$ and $\g_0 = \a \oplus \m$, it follows that $\h \cap \g_0$ projects surjectively onto $\a$. 

Let  $\alpha$ be a positive root and let $v\in\a$  be such that $\alpha(v)\neq 0$. Then $\operatorname{ad}_ {v}(\g_\alpha)=\g_\alpha$. Now, let $m_v\in\m$ such that
$v+m_v\in\h\cap\g_0$. If $m_v=0$, then  $\operatorname{ad}_ {v+m_v}(\g_\alpha)=\operatorname{ad}_v(\g_\alpha)=\g_\alpha$. Otherwise, since $\operatorname{ad}_ {m_v}$ is antisymmetric and $\operatorname{ad}_v(\g_\alpha)=\g_\alpha$, it still follows that  $\operatorname{ad}_ {v+m_v}(\g_\alpha)=\g_\alpha$. Thus $\g_\alpha\subset \h$ and hence $\hat{\a}\oplus\g_+\subset \h$. 
On the other hand, since $\hat{\m} \subseteq \h$, we have  $\operatorname{ad}_{\g_{-\delta}}(\g)\subset \hat{\a}\oplus \hat{\m}\oplus\g_{+}\subset \h$. 

Now, let $X$ be a non trivial element of $\g_{-\delta}$. Then $d_{1}e^{X}$ acts trivially on $\g/\h$, whereas $e^{X}$ is non-trivial. We are therefore exactly in the setting of Frances-Melnick \cite[Theorem~1.8]{FM}, which implies that our space is conformally flat.

\subsection{Fourth case: $\hat{\m} \nsubseteq \h$ }
Using results from \cite{BDRZ1} and the classification of simple real Lie algebras, we are going to identify the Lie algebra $\mathfrak{g}$. We have  that the triple $(\g_{\Gamma_{\delta}},\h \cap \g_{\Gamma_{\delta}}, \langle,\rangle _{\mathfrak{g}_{\Gamma_{\delta}}})$ satisfies the heredity principale. Since $\hat{\m} \nsubseteq \h$, this corresponds to the case described in \cite[Subsection~4.2, Subsection~4.3]{BDRZ1}. It then follows  that: 
\begin{prop} We have:
\label{prop28}
\begin{itemize}
  \item $\hat{\g}(\delta)$ is isomorphic to $\so(1,4)$;
  \item $\g_{\delta} \cap \h = 0$ and $\hat{\m} \cap \h = 0$;
  \item $\g_{\delta} \simeq  \mathbb{R}^{3}$;
  \item $\hat{\m} \simeq \so(3)$ is an ideal of $\m$.
\end{itemize}
\end{prop}

Therefore, since \(\mathfrak{g}_{\delta}\) is paired with $
\mathfrak{g}_{0}=\mathfrak{a}\oplus\hat{\mathfrak{m}}\oplus\mathfrak{g}_{c}\oplus\bigl(\mathfrak{m}\cap Z(\mathfrak{g}_{\Gamma_{\delta}})\bigr)
$, and non-degenerately paired precisely with $\hat{m}$, we deduce that $\g_{0}\cap\h$ projects surjectively on $\a\oplus \g_{c}\oplus\left(  \mathfrak{m} \cap Z(\mathfrak{g}_{\Gamma_{\delta}})\right)$, and \begin{equation}\label{equa93}
\g_{0}/\g_{0}\cap\h\simeq \hat{m}
\end{equation}
This shows, in particular, that---up to modification---one may disregard the compact semisimple part~$\mathfrak{g}_c$ and assume that~$\mathfrak{g}$ is semisimple of non compact type. Combined with Lemma~\ref{Lemma1}, this actually yields:
\begin{lemm}
\label{simplenoncompact}
Up to modification, \(\mathfrak{g}\) is simple of non-compact type.
\end{lemm}

Thanks to Lemma~\ref{simplenoncompact}, from now on and until the end of the paper, we assume that~$\mathfrak{g}$ is simple of non-compact type. According to~\cite[Theorem~6.105]{K}, the simple non-compact real Lie algebras are:

\medskip

\noindent
\textbf{1.} The non-exceptional complex simple Lie algebras viewed as real:
$\sl(n,\C)$, $\so(n,\C)$, $\sp(n,\C)$.

\smallskip

\noindent
\textbf{2.} The exceptional complex simple Lie algebras viewed as real:
$E_{6}$, $E_{7}$, $E_{8}$, $G_{2}$, $F_{4}$.

\smallskip

\noindent
\textbf{3.} The algebras $\su(p,q)$ with $p \ge q > 0$ and $p+q \ge 2$.

\smallskip

\noindent
\textbf{4.} The algebras $\so(p,q)$ with $p > q > 0$, $p+q$ odd, and $p+q \ge 5$.

\smallskip

\noindent
\textbf{5.} The algebras $\so(p,q)$ with $p \ge q > 0$, $p+q$ even, and $p+q \ge 8$.

\smallskip

\noindent
\textbf{6.} The algebras $\sp(p,q)$ with $p \ge q > 0$ and $p+q \ge 3$.

\smallskip

\noindent
\textbf{7.} $\sp(n,\R)$ for $n \ge 3$.

\smallskip

\noindent
\textbf{8.} $\sl(n,\R)$ for $n \ge 3$.

\smallskip

\noindent
\textbf{9.} $\sl(n,\H)$ for $n \ge 2$.

\smallskip

\noindent
\textbf{10.} $\so^{*}(2n)$ for $n \ge 4$.

\smallskip

\noindent
\textbf{11.} The twelve exceptional, non-complex, non-compact real simple Lie algebras:
$\mathrm{E}_{\mathrm{I}}, \mathrm{E}_{\mathrm{II}}, \mathrm{E}_{\mathrm{III}},
\mathrm{E}_{\mathrm{IV}}, \mathrm{E}_{\mathrm{V}}, \mathrm{E}_{\mathrm{VI}},
\mathrm{E}_{\mathrm{VII}}, \mathrm{E}_{\mathrm{VIII}}, \mathrm{E}_{\mathrm{IX}},
\mathrm{E}_{\mathrm{FI}}, \mathrm{E}_{\mathrm{FII}}, \mathrm{E}_{\mathrm{G}}$.

A quick inspection of the Data for simple Lie algebras  in \cite[Appendix~C]{K} shows the following:
\begin{itemize}
\item All complex simple Lie algebras have root spaces of real dimension $2$. 
\item The Lie algebras $\mathfrak{sp}(n, \mathbb{R})$ and $\mathfrak{sl}(n, \mathbb{R})$ are split; hence their root spaces are one-dimensional.
\item  The Lie algebra $\mathfrak{sl}(n, \mathbb{H})$ has root spaces of dimension $4$. 
\item For $\mathfrak{so}(p, q)$ (with $p\geq q>1$), $\mathfrak{so}^*(2n)$ and $\mathfrak{su}(p, q)$, the root space corresponding to the minimal root is also one-dimensional.
\item Among the exceptional non-complex non-compact simple Lie algebras $\operatorname{E}{\operatorname{I}}$, $\operatorname{E}{\operatorname{V}}$, $\operatorname{E}{\operatorname{VIII}}$, $\operatorname{F}{\operatorname{I}}$, and $\operatorname{G}$ all have $\m = 0$. In contrast, for $\operatorname{E}{\operatorname{II}}$, $\m = \mathbb{R}^2$, while for $\operatorname{E}{\operatorname{IV}}$, $\operatorname{E}{\operatorname{VII}}$, $\operatorname{E}{\operatorname{IX}}$, $\operatorname{F}{\operatorname{II}}$, $\m $ is a simple Lie algebra. 
\item In the case of  $\operatorname{E}{\operatorname{III}}$, $\m=\mathfrak{su}(4)\oplus\mathbb{R}$.
\item Finally, the real rank-one subalgebras of  $\operatorname{E}{\operatorname{IV}}$  are $\mathfrak{so}(1, 2)$ and $\mathfrak{so}(1, 5)$.
\end{itemize}
Hence, combined with Proposition~\ref{prop28}, we conclude:

\begin{prop}
The only possible case for $\mathfrak{g}$ is $\sp(p,q)$.
\end{prop}

\section{Conformal Flatness}
\label{sect4}
From now on, we assume that \(\mathfrak{g} = \mathfrak{sp}(p, q)\). The root system of \(\mathfrak{sp}(p, q)\) is of type \((BC)_p\) when \(p < q\), and of type \(C_p\) when \(p = q\). Up to isomorphism, it can be described--like all other root systems--using the canonical basis \((e_1, \ldots, e_p)\) of \(\mathbb{R}^p\) (See for instance \cite[Appendix~C]{K}). We will assume, up to isomorphism, that the root system $\Delta$ of $\mathfrak{g}$ is the canonical root system endowed with its canonical order.
\begin{prop}
\label{porpro}
$\g_{+}\subset \h$.
\end{prop}
\begin{proof}
Let $\alpha$ be a positive root, and take $v\in\mathfrak{a}$ such that $\alpha(v)\neq 0$. 
Let $m_v\in\mathfrak{m}$ be such that $v+m_v\in\mathfrak{h}\cap\mathfrak{g}_0$. 
As in Subsection~\ref{subsection5.2}, one obtains $
\operatorname{ad}_{\,v+m_v}(\mathfrak{g}_\alpha)=\mathfrak{g}_\alpha
$. It then  follows that $\mathfrak{g}_\alpha\subset \mathfrak{h}$, and hence $\mathfrak{g}_{+}\subset \mathfrak{h}$.

\end{proof}

Let $\alpha$ be a root in $\Delta \setminus \mathbb{R}\delta$, and let $\hat{\g}(\alpha, \delta)$ denote the corresponding non-compact semisimple subalgebra of rank-two. 

\begin{lemm}
\label{prop289}
The semisimple Lie algebra $\hat{\g}(\alpha, \delta)$ is not  of type~$\mathrm{A}_{2}$.
\end{lemm}
\begin{proof}

From Equation~\ref{eq:commutantGamma1} and Proposition~\ref{prop28}, we know that $\mathfrak{g}_{\delta} = \mathbb{R}^{3}$ is a root space of $\hat{\mathfrak{g}}(\alpha, \delta)$.
Suppose now that $\hat{\mathfrak{g}}(\alpha, \delta)$ were of type~$\mathrm{A}{2}$.
Then, by looking in the Data for simple Lie algebras \cite[Appendix C]{K}, it would be isomorphic (as a real Lie algebra) to one of $\mathfrak{sl}(3,\mathbb{C})$, $\mathfrak{sl}(3,\mathbb{R})$, or $\mathfrak{sl}(3,\mathbb{H})$.
None of these algebras have root spaces of dimension~3.
This is a contradiction.
\end{proof}

Consequently, we obtain:
\begin{prop}
\label{prrrrrroooooopppp1}
$\a\subset \h$.
\end{prop}
\begin{proof}
Let $\alpha$ be  a  root in $\Delta \setminus \mathbb{R}\delta$. By Lemma~\ref{prop289}, the rank-two semisimple subalgebra $\hat{\g}(\alpha, \delta)$ is not of type~$\mathrm{A}{2}$. By a simple examination, we can see that in each of the remaining rank-two types, there is a negative root $\beta\notin \mathbb{R}\delta$ such that $\delta-\beta$ is not a root. Indeed, if $\hat{\g}(\alpha, \delta)$ were of type $\mathrm{B}{2}$ or $\mathrm{D}{2}$, then $\delta=-e_{1}-e_{2}$ and one may choose $\beta=e_{2}-e_{1}$.  
If it were of type $\mathrm{G}{2}$, then $\delta=-2e_{3}+e_{1}+e_{2}$ and one may choose $\beta=e_{2}-e_{1}$.  
Finally, if it were of type $(\mathrm{BC})_{2}$, then $\delta=-2e_{1}$ and one may choose $\beta=-e_{2}$.

It follows that $\g_{\pm\beta}\subset\h$, and hence $[\g_{\beta},\g_{-\beta}]\subset \h$. In particular, $H_{\beta}\in \h$ , and therefore the entire Cartan subalgebra of $\hat{\g}(\alpha, \delta)$ lies in $\h$. Since this is true for every $\alpha$, we get that $\a\subset \h$. 
\end{proof}

\begin{prop} \label{p}
We have:
\begin{enumerate}
  \item If $p < q$, then:
  \[
  \mathfrak{g}/\mathfrak{h} \simeq  \mathfrak{g}_{-2e_1} \oplus \mathfrak{g}_{-e_1} \oplus \bigoplus_{i > 1} \mathfrak{g}_{-(e_1 \pm e_i)} \oplus \mathfrak{so}(3).
  \]
  
  \item If $p = q$, then:
  \[
  \mathfrak{g}/\mathfrak{h} \simeq  \mathfrak{g}_{-2e_1} \oplus \bigoplus_{i > 1} \mathfrak{g}_{-(e_1 \pm e_i)} \oplus \mathfrak{so}(3).
  \]
\end{enumerate}
\end{prop}

\begin{proof}
Without loss of generality, we may assume that \( p < q \). The proof in the case $p=q$ is similar. Note that by Proposition~\ref{prop28}, we have \( \delta = -2e_1 \), and hence 
\begin{equation}
\mathfrak{g}_{-2e_1} \cap \mathfrak{h} 
= \mathfrak{g}_{\delta} \cap \mathfrak{h} 
= 0 .
\end{equation}

First, on the one hand, for every \( 1 < i < j \), \( \delta + (e_i \pm e_j) = -2e_1 + e_i \pm e_j \), \( \delta + e_i = -2e_1 + e_i \), and \( \delta + 2e_i = -2e_1 + 2e_i \) are not roots. Hence, by the pairing condition, we conclude that\begin{equation}\label{equa5}
\forall\, 1 < i < j, \quad 
\mathfrak{g}_{-(e_i \pm e_j)}, \; 
\mathfrak{g}_{-e_i}, \; 
\mathfrak{g}_{-2e_i} \;\subset\; \mathfrak{h}.
\end{equation}\\
On the other hand, by Propositions~\ref{porpro}, we have
 \begin{equation}\label{equa7}
\forall\, 1 \leq i < j \leq p, \quad 
\mathfrak{g}_{e_i \pm e_j}, \;
\mathfrak{g}_{e_i}, \;
\mathfrak{g}_{2e_i}
\;\subset\; \mathfrak{h}.
\end{equation}

Second, for all \( j > 1 \), we have \( \mathfrak{g}_{-(e_1 \pm e_j)} \not\subset \mathfrak{h} \). Indeed, suppose for contradiction that there exists \( j > 1 \) such that \( \mathfrak{g}_{-(e_1 + e_j)} \subset \mathfrak{h} \) or \( \mathfrak{g}_{-(e_1 - e_j)} \subset \mathfrak{h} \). Since the roots \( -(e_1 + e_j) \) and \( -(e_1 - e_j) \) are paired, it follows that both are contained in \( \mathfrak{h} \). Then, by \cite{Kraljevic}, we would have
$
\mathfrak{g}_{-2e_1} = [\mathfrak{g}_{-(e_1 + e_j)}, \mathfrak{g}_{-(e_1 - e_j)}] \subset \mathfrak{h},
$
which contradicts the fact that \( \mathfrak{g}_{-2e_1} \cap \mathfrak{h} = 0 \).\\
Similarly, \( \mathfrak{g}_{-e_1} \not\subset \mathfrak{h} \). Indeed, if \( \mathfrak{g}_{-e_1} \subset \mathfrak{h} \), then by \cite{Kraljevic}, we would have
$
\mathfrak{g}_{-2e_1} = [\mathfrak{g}_{-e_1}, \mathfrak{g}_{-e_1}] \subset \mathfrak{h},
$
again yielding a contradiction.

Third, we claim that 
\begin{equation}
\label{equa8}
\forall\, j > 1, \quad 
\mathfrak{g}_{-(e_1 \pm e_j)} \cap \mathfrak{h} = 0 .
\end{equation}
Suppose, to the contrary, that there exists \( x \in \mathfrak{g}_{-(e_1 \pm e_j)} \cap \mathfrak{h} \) with \( x \neq 0 \). Then, by the Lemma of \cite{Kraljevic}, for any \( k \neq 1, j \), we have:
\[
\mathfrak{g}_{-(e_1 + e_k)} = [x, \mathfrak{g}_{-(e_k \mp e_j)}] \subset [\mathfrak{h}, \mathfrak{h}] \subset \mathfrak{h},
\]
which contradicts the previous conclusion that \( \mathfrak{g}_{-(e_1 + e_k)} \not\subset \mathfrak{h} \). \\
A similar argument shows that \begin{equation}
\label{equa9}
\mathfrak{g}_{-e_1} \cap \mathfrak{h} = 0 .
\end{equation}
Indeed, suppose there exists \( x \in \mathfrak{g}_{-e_1} \cap \mathfrak{h} \) with \( x \neq 0 \). Then, by the Lemma of \cite{Kraljevic}, we obtain
\[
\mathfrak{g}_{-e_1 - e_2} = [x, \mathfrak{g}_{-e_2}] \subset [\mathfrak{h}, \mathfrak{h}] \subset \mathfrak{h},
\]
which contradicts again our previous conclusion that \( \mathfrak{g}_{-e_1 - e_2} \not\subset \mathfrak{h} \).\\
Finally, by Proposition~\ref{prop28}, we have
\begin{equation}
\label{equa10}
\hat{\mathfrak{m}} \cap \mathfrak{h} 
= \mathfrak{so}(3) \cap \mathfrak{h} 
= 0 .
\end{equation}

Fourth, we claim that \begin{equation}
\label{equa11}
\left( 
\mathfrak{g}_{-2e_1} 
\oplus \mathfrak{g}_{-e_1} 
\oplus_{j>1} \mathfrak{g}_{-(e_1 \pm e_j)} 
\oplus \hat{\mathfrak{m}}
\right) 
\cap \mathfrak{h} = 0 .
\end{equation}
By contradiction, let $0\neq x \in\left( \mathfrak{g}_{-2e_1}\oplus \mathfrak{g}_{-e_1}\oplus_{j>1} \mathfrak{g}_{-(e_1 \pm e_j)}\oplus\hat{\mathfrak{m}}\right)   \cap \mathfrak{h}$. Assume first, that the projection of $x$ on $\mathfrak{g}_{-2e_1}$ is non nul. On the one hand, by the Lemma of \cite{Kraljevic}, the projection of $[\mathfrak{a},[x,\mathfrak{g}_{e_1}]]$ on $\mathfrak{g}_{-e_1}$ is onto. On the other hand, by Proposition~\ref{prrrrrroooooopppp1}, $[\mathfrak{a},[x,\mathfrak{g}_{e_1}]]\subset \mathfrak{h} $. Now, since for every $j>1$, $\mathfrak{g}_{\pm e_j}\subset \mathfrak{h}$, we obtain that $\mathfrak{g}_{-e_1}\subset \mathfrak{h}$ which clearly contradicts Equation~\ref{equa9}.\\
Let $\alpha$ be a root in $\lbrace e_{1}, e_{1}+ e_{j}, j>1\rbrace\cup \lbrace  e_{1}- e_{j}, j>1\rbrace$ such that the projection of $x$ on $\mathfrak{g}_{-\alpha}$ is non nul. Note that each of the two family  $\lbrace e_{1}, e_{1}+ e_{j}, j>1\rbrace$ and $\lbrace  e_{1}- e_{j}, j>1\rbrace$ is linearly independent. Thus, there exists two non nuls elements
\[
 a_{1} \in 
\bigcap_{\beta \in \{ e_{1}, \, e_{1}+ e_{j} \mid j>1 \} \setminus \{\alpha\}}
\operatorname{Ker}(\beta)
\]
and
\[
 a_{2} \in 
\bigcap_{\beta \in \{ e_{1}- e_{j} \mid j>1 \} \setminus \{\alpha\}}
\operatorname{Ker}(\beta).
\]
Therefore, again by Proposition~\ref{prrrrrroooooopppp1}, $0\neq [a_{2},[a_{1},x]]\in \mathfrak{h}\cap \mathfrak{g}_{-\alpha}$ which contradicts Equation~\ref{equa8} (or Equation~\ref{equa9}).\\
Now, we are left with $x\in \hat{\mathfrak{m}}\cap \mathfrak{h}$ which contradicts Equation~\ref{equa10}.

In summary, combining Equations~\eqref{equa93}, \eqref{equa5}, \eqref{equa7}, and \eqref{equa11} give us:
\[
\g/\mathfrak{h} \simeq \mathfrak{g}_{-2e_1} \oplus \mathfrak{g}_{-e_1} \oplus \bigoplus_{i > 1} \mathfrak{g}_{-(e_1 \pm e_i)} \oplus \mathfrak{so}(3).
\]

\end{proof}

Next, we will show that the conformal structure is unique:
\begin{prop}
\label{prop280}
If the conformal structure $\langle .,.\rangle$ exists then it is unique.
\end{prop}
\begin{proof}
By Proposition~\ref{p}, we have that $\mathfrak{g}/\mathfrak{h}\simeq \mathfrak{g}_{-2e_{1}}\oplus\mathfrak{g}_{-e_{1}} \bigoplus_{j>1} \mathfrak{g}_{-(e_{1}\pm e_{j})}\oplus \mathfrak{so}(3)$. Note that:
\begin{enumerate}
\item $\mathfrak{g}_{-2e_{1}}$ is paired with $\mathfrak{so}(3)$;
\item $\mathfrak{g}_{-e_{1}}$ is paired with itself; and
\item for every $j>1$, $\mathfrak{g}_{-(e_{1}+ e_{j})}$ is paired with $\mathfrak{g}_{-(e_{1}-e_{j})}$.
\end{enumerate}
Let $\alpha$ be a negative root such that $\mathfrak{g}_{-\alpha} \cap \mathfrak{h} = 0$, and fix an element $0 \neq z \in \mathfrak{g}_\delta$. Then by the Lemma of \cite{Kraljevic}, we have
\[
[\mathfrak{g}_{\alpha}, z] = \mathfrak{g}_{\delta + \alpha}.
\]
Hence, every $y \in \mathfrak{g}_{\delta + \alpha}$ can be written as $y = \operatorname{ad}_p(z)$ for some $p \in \mathfrak{g}_{\alpha}$. Since $p$ acts isometrically, we get by applying Equation~\ref{Equa2} to $p$, $x\in \mathfrak{g}_{-\alpha}$, and $z$:
\[
\langle x, y \rangle = -\langle \operatorname{ad}_p(x), z \rangle.
\]
This implies that the conformal structure $\langle \cdot, \cdot \rangle$ is entirely determined by its restriction to $\mathfrak{g}_{\Gamma_\delta}$, which is  unique by \cite[Proposition~4.2]{BDRZ1}.

\end{proof}

Now, in order to complete the proof of our main theorem, we need to show that we already have an example of such a situation. To do so, let us return to our example of the action of \(\Sp(p,q)\) on \(\Ein^{4p-1, 4q-1}\). In this case, we have that $\mathfrak{p}^{\mathbb{H}}$ is a vector-space direct sum $\mathfrak{p}^{\mathbb{H}}=\mathfrak{p}^{\mathbb{R}}\oplus \so(3)$ and $
\mathfrak{a} \oplus \mathfrak{g}_{+} \subseteq \mathfrak{p}^{\mathbb{R}}.
$ By Proposition~\ref{proppppppppppppppp}, $\m\nsubseteq \mathfrak{p}^{\mathbb{R}}.$ Actually we have more: $\hat{\m} \nsubseteq \h$. If not, then $\mathfrak{g}_{-2e_1} \cap \mathfrak{p}^{\mathbb{R}} = \mathfrak{g}_{\delta} \cap \mathfrak{p}^{\mathbb{R}} \neq 0
$. But, by \cite[Lemma]{Kraljevic}, we have that for every $x\in \mathfrak{g}_{-2e_1} \cap \mathfrak{p}^{\mathbb{R}}$, $
\mathfrak{g}_{-e_{1}}=[\mathfrak{g}_{e_{1}}, x]\subset [\mathfrak{p}^{\mathbb{R}}, \mathfrak{p}^{\mathbb{R}}]\subset \mathfrak{p}^{\mathbb{R}} 
$. Thus $\mathfrak{g}_{-2e_{1}}=\mathfrak{g}_{\delta}\subset \mathfrak{p}^{\mathbb{R}} $, which contradicts the fact that $\m\nsubseteq \mathfrak{p}^{\mathbb{R}}.$

Therefore, combining Proposition~\ref{prop28} together with the example given in section~\ref{secexample} shows that:
\begin{prop}
If $\hat{\m} \nsubseteq \h$, then $M$ is, up to finite cover, conformally equivalent to $\Ein^{4p-1, 4q-1}$. In particular, $M$ is conformally flat.
\end{prop}

\bibliographystyle{plain}

\bibliography{bibliosemisimple}

\end{document}